\documentclass[11pt]{article}

\usepackage{amssymb, amsmath, amsthm, graphicx}
\usepackage[left=1in,top=1in,right=1in]{geometry}
\usepackage{enumitem}
\usepackage{subfigure}

      \theoremstyle{plain}
      \newtheorem{theorem}{Theorem}[section]
      \newtheorem{lemma}[theorem]{Lemma}
            
      \newtheorem{corollary}[theorem]{Corollary}

      \theoremstyle{definition}

      \theoremstyle{remark}

\def\twr{\mbox{\rm twr}}

\def\sgn{\mbox{\rm sgn}}
\def\ot{\mbox{\rm OT}}

\title{A note on order-type homogeneous point sets}
\author{ Andrew Suk\thanks{Massachusetts Institute of Technology, Cambridge, MA. Supported by an NSF Postdoctoral
Fellowship and by Swiss National Science Foundation Grant 200021-125287/1. Email: {\tt asuk@math.mit.edu}.}}

\begin{document}
\maketitle

\medskip

\begin{abstract}

Let $\ot_d(n)$ be the smallest integer $N$ such that every $N$-element point sequence in $\mathbb{R}^d$ in general position contains an order-type homogeneous subset of size $n$, where a set is \emph{order-type homogeneous} if all $(d+1)$-tuples from this set have the same orientation.  It is known that a point sequence in $\mathbb{R}^d$ that is order-type homogeneous, forms the vertex set of a convex polytope that is combinatorially equivalent to a cyclic polytope in $\mathbb{R}^d$.  Two famous theorems of Erd\H os and Szekeres from 1935 imply that $\ot_{1}(n) = \Theta(n^2)$ and $\ot_2(n) = 2^{\Theta(n)}$.  For $d\geq 3$, we give new bounds for $\ot_d(n)$.  In particular:

\begin{itemize}
\item We show that $\ot_3(n) = 2^{2^{\Theta(n)}}$, answering a question of Eli\'a\v{s} and Matou\v{s}ek.

\item For $d\geq 4$, we show that $\ot_d(n)$ is bounded above by an exponential tower of height $d$ with $O(n)$ in the topmost exponent.

\end{itemize}

\end{abstract}

\section{Introduction}

In their classic paper \cite{es}, Erd\H os and Szekeres proved the following two well-known results.

 \begin{theorem}
 \label{es1}
 For $N = (n-1)^2 + 1$, let $P = (p_1,...,p_N) \subset \mathbb{R}$ be a sequence of $N$ distinct real numbers.  Then $P$ contains a subsequence $(p_{i_1},...,p_{i_n})$, $i_1 < \cdots < i_n$, such that either $p_{i_1} < p_{i_2} < \cdots < p_{i_n}$ or $p_{i_1} > p_{i_2} > \cdots > p_{i_n}$.

 \end{theorem}

\noindent In fact, there are now at least 6 different proofs of Theorem \ref{es1} (see \cite{steele}).  Notice that the point sequence $(p_{i_1},...,p_{i_n})$ obtained from Theorem \ref{es1} has the property that either $p_{i_k} - p_{i_j} > 0$ for every pair $j,k$ such that $1 \leq j < k\leq n$, or $p_{i_k} - p_{i_j} < 0$ for every pair $j,k$ such that $1 \leq j < k\leq n$.  The other well-known result from \cite{es} is the following theorem, which is often referred to as the Erd\H os-Szekeres cups-caps Theorem (see also \cite{mosh}).

\begin{theorem}
\label{es2}
For $N =  {2n - 4\choose n-2} + 1$, let $P = (p_1,...,p_N)$ be a sequence of $N$ points in the plane such that no 2 share a common first coordinate, $P$ is ordered by increasing first coordinate, and no 3 points lie on a line.  Then $P$ contains a subsequence $(p_{i_1},...,p_{i_n})$, $i_1 < \cdots < i_n$, such that the slopes of the lines $p_{i_j}p_{i_{j+1}}$, $j = 1,2,...,n-1$, are increasing or decreasing.
\end{theorem}

\noindent See Figure \ref{cupcapex}.  Again, notice that the point sequence $(p_{i_1},...,p_{i_n})$ obtained from Theorem \ref{es2} has the property that either every triple has a clockwise orientation, or every triple has a counterclockwise orientation.

\begin{figure}[h]
\begin{center}
\includegraphics[width=280pt]{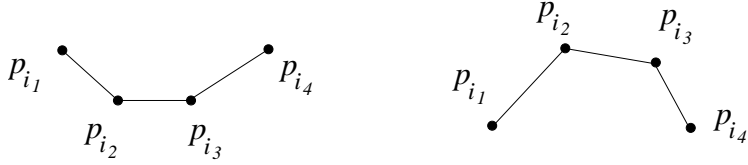}
  \caption{Points $p_{i_1},p_{i_2},p_{i_3},p_{i_4}$ obtained from Theorem \ref{es2}.}\label{cupcapex}
 \end{center}
\end{figure}

The preceding discussion generalizes in a natural way to point sequences in $\mathbb{R}^d$ in general position.  A point set $P$ in $\mathbb{R}^d$ is in \emph{general position}, if no $d+1$ members lie on a common hyperplane, and no 2 members share the same $i$-th coordinate for $1\leq i \leq d$.

Let $P = (p_1,...,p_N)$ be an $N$-element point sequence in $\mathbb{R}^d$ in general position.  For $i_1 < i_2 < \cdots < i_{d+1}$, the \emph{orientation} of the $(d+1)$-tuple $(p_{i_1},p_{i_2},...,p_{i_{d+1}}) \subset P$ is defined as the sign of the determinant of the unique linear mapping $A$ that sends the $d$ vectors $p_{i_2}-p_{i_1},p_{i_3}-p_{i_1},...,p_{i_{d+1}} - p_{i_1}$, to the standard basis $e_1,e_2,...,e_{d}$.  Geometrically, if $h\subset \mathbb{R}^d$ is the hyperplane spanned by $p_{i_1},...,p_{i_d}$, then the orientation of the $(d+1)$-tuple $(p_{i_1},p_{i_2},...,p_{i_{d + 1}})$ tells us on which side of the hyperplane $h$ the point $p_{i_{d+1}}$ lies.

The \emph{order type} of $P= (p_1,p_2,...,p_N)$ is the mapping $\chi:{P \choose d + 1}\rightarrow \{+1,-1\}$ (positive orientation, negative orientation), assigning each $(d+1)$-tuple of $P$ its orientation.  Hence for $p_i = (a_{i,1},a_{i,2},...,a_{i,d}) \in \mathbb{R}^d$,

$$\chi(\{p_{i_1},p_{i_2},...,p_{i_{d + 1}}\})  =  \sgn\det\left(\begin{array}{cccc}
                                                           1 & 1 & \cdots & 1 \\
                                                           a_{i_1,1} & a_{i_2,1} & \cdots & a_{i_{d+1},1} \\
                                                           \vdots & \vdots & \vdots & \vdots \\
                                                           a_{i_1,d} & a_{i_2,1} & \cdots & a_{i_{d+1},1}
                                                         \end{array}\right).$$

\noindent Therefore, two $N$-element point sequences $P$ and $Q$ have the same order type if they are ``combinatorially equivalent."  We say that a point sequence in $\mathbb{R}^d$ is \emph{order-type homogeneous}, if all $(d+1)$-tuples have the same orientation.  See \cite{mat} and \cite{gp} for more background on order types.

Order-type homogeneous point sets exhibit several fascinating combinatorial and algebraic properties.  Recall that an $n$-vertex \emph{cyclic polytope} in $\mathbb{R}^d$ is the convex hull of $n$ points on the moment curve $\gamma = \{(t,t^2,...,t^d): t \in \mathbb{R}\}\subset \mathbb{R}^d$.  A well-known folklore states that a point sequence $P$ in $\mathbb{R}^d$ that is order-type homogeneous forms the vertex set of a convex polytope which is combinatorially equivalent to the cyclic polytope in $\mathbb{R}^d$ (see \cite{mat} Exercise 5.4.3 and \cite{duchet}).  A classic result of McMullen \cite{mc} states that among all $d$-dimensional convex polytopes with $n$ vertices, the cyclic polytope maximizes the number of faces of each dimension.

Following the notation of Eli\'a\v{s} and Matou\v{s}ek \cite{matousek} we define $\ot_d(n)$ to be the smallest integer $N$ such that any $N$-element point sequence in $\mathbb{R}^d$ in general position, contains an $n$-element subsequence that is order-type homogeneous.  For dimension one, an order-type homogeneous sequence in $\mathbb{R}$ is just an increasing or decreasing sequence of real numbers.  Hence, Theorem \ref{es1} implies that $\ot_1(n) \leq (n-1)^2 + 1$.  On the other hand, a simple construction from \cite{es} shows that $\ot_1(n) = (n-1)^2 + 1$.  For dimension two, an order-type homogeneous point sequence in $\mathbb{R}^2$ is a planar point sequence in convex position, which appear in either clockwise or counterclockwise order along the boundary of their convex hull (see \cite{hubard}).  By combining Theorems \ref{es1} and \ref{es2}, one can show\footnote{We write $f(n) = O(g(n))$ if $|f(n)| \leq c|g(n)|$ for some fixed constant $c$ and for all $n \geq 1$; $f(n) = \Omega(g(n))$ if $g(n) = O(f(n))$; and $f(n) = \Theta(g(n))$ if both $f(n) = O(g(n))$ and $f(n) = \Omega(g(n))$ hold.} that $\ot_2(d) \leq 2^{O(n)}$.  On the other hand, a famous construction of Erd\H os and Szekeres \cite{es22} on point sets in the plane with no large convex subset, shows that $\ot_2(n) = 2^{\Theta(n)}$.

For several decades, the best known upper bound on $\ot_d(n)$ for fixed $d\geq 3$ was obtained by applying Ramsey numbers\footnote{The \emph{Ramsey number} $R_k(n)$ is the least integer $N$ such that every red-blue coloring of all unordered $k$-tuples of an $N$-element set contains either a red set of size $n$ or a blue set of size $n$, where a set is called red (blue) if all $k$-tuples from this set are red (blue).} (see \cite{matousek,grunbaum,duchet} and \cite{graham,erdos2,erdos3,rado}).  This implies

  $$\ot_d(n) \leq \twr_{d+1}(O(n)),$$

  \noindent where the tower function $\twr_k(x)$ is defined by $\twr_1(x) = x$ and $\twr_{i + 1} = 2^{\twr_i(x)}$.  Recently, Conlon et al. \cite{suk} improved this upper bound to $\twr_{d}(n^{c_d})$, where $c_d$ is exponential in a power of $d$.    Our main result establishes a further improvement on the upper bound of $\ot_d(n)$.

\begin{theorem}
\label{main1}
For fixed $d\geq 2$, we have $\ot_d(n) \leq \twr_{d}(O(n)).$

\end{theorem}

A recent result of Eli\'a\v{s} and Matou\v{s}ek \cite{matousek} shows that $\ot_3(n) \geq 2^{2^{\Omega(n)}}$.  Hence as an immediate corollary to Theorem \ref{main1}, we have obtained a reasonably tight bound on $\ot_3(n)$.

\begin{corollary}
\label{main2}
For dimension three, we have $\ot_3(n) = 2^{2^{\Theta(n)}}$.

\end{corollary}

\section{Proof of Theorem \ref{main1}}

In this section we will prove Theorem \ref{main1}.  First we will introduce some notions.  For $p,q\in \mathbb{R}^d$ where $p = (a_1,...,a_d)$ and $q = (b_1,...,b_d)$, we say that $p$ lies \emph{above} $q$ if $a_d >  b_d$.  Likewise, we say that $p$ lies \emph{below} $q$ if $a_d < b_d$.

Recall the following lemma on arrangements of hyperplanes (see \cite{mat}).

 \begin{lemma}
For $m\geq d \geq 2$, the number of cells ($d$-faces) in an arrangement of $m$ hyperplanes in $\mathbb{R}^d$ is at most $m^d$.
\end{lemma}

We now establish the following recursive formula for $\ot_d(n)$.

\begin{lemma}
\label{hyperplane}
For $M = \ot_{d-1}(n-1)$ and $d \geq 2$,

$$\ot_d(n)  \leq 2^{4d^2M\log M}.$$

\end{lemma}

\begin{proof}
Let $P =(p_1,...,p_N)$ be a sequence of $N = 2^{4d^2M\log M}$ points in $\mathbb{R}^d$ in general position, and let $\chi:{P\choose d+1}\rightarrow \{+1,-1\}$ be the order type of $P$.  In what follows, we will recursively construct a sequence of points $q_1,...,q_r$ from $P$ and a subset $S_r \subset P$, where $r = d-1,...,2M$, such that the following hold.

 \begin{enumerate}[label={(\arabic*)}]

 \item For $i<j$, $q_i$ comes before $q_j$ in the original ordering and every point in $S_r$ comes after $q_r$ in the original ordering.

 \item Every $d$-tuple $(q_{i_1},...,q_{i_{d}})\subset \{q_1,...,q_{r}\}$ with $i_1 < i_2 < \cdots <i_{d}$ has the property that either $\chi(q_{i_1},...,q_{i_{d}}, q)  =+1$ for every point $q \in \{q_j:i_{d} < j \leq r\} \cup S_r$, or $\chi(q_{i_1},...,q_{i_{d}}, q) = -1$ for every point $q \in \{q_j:i_{d} < j \leq r\} \cup S_r$.

     \item We have $|S_r| \geq \frac{N}{((r-1)!)^{d^2}}-r.$
 \end{enumerate}

We start by selecting the $d-1$ points $\{q_1,...,q_{d-1}\}  = \{p_1,...,p_{d-1}\}$ from $P$ and setting $S_{d-1} = P \setminus\{p_1,...,p_{d-1}\}$.  After obtaining $\{q_1,...,q_r\}$ and $S_r$, we define $q_{r+1}$ and $S_{r+1}$ as follows.  Let $q_{r+1}$ be the smallest element in $S_r$.

In order to obtain (2), we only need to consider the $d$-tuples from $\{q_1,...,q_{r+1}\}$ that include the last point $q_{r+1}$.  Notice that each $(d-1)$-tuple $T=\{q_{i_1},q_{i_2},...,q_{i_{d-1}}\} \subset \{q_1,...,q_r\}$ gives rise to a hyperplane spanned by the points $T\cup \{q_{r+1}\}$.  Let $H_r$ be the set of these ${r\choose d-1}$ hyperplanes.  By Lemma \ref{hyperplane}, the number of cells in the arrangement of $H_r$ is at most

$${r \choose d-1}^d \leq r^{d^2}.$$

\noindent By the pigeonhole principle and since $P$ is in general position, there exists a cell ($d$-face) $\Delta\subset \mathbb{R}^{d}$ that contains at least $(|S_r| - 1)/r^{d^2}$ points of $S_r$.  Hence, for any fixed $d$-tuple $(q_{i_1},...,q_{i_{d}}) \subset \{q_1,...,q_{r+1}\}$, we have either

$$\chi(q_{i_1},...,q_{i_{d}},p) = +1 \hspace{1cm}\forall p \in \Delta\cap S_r\setminus\{q_{r+1}\}$$

\noindent or

$$\chi(q_{i_1},...,q_{i_{d}},p) = -1 \hspace{1cm}\forall p \in \Delta\cap S_r\setminus\{q_{r+1}\}.$$

\noindent Set $S_{r+1} = \Delta\cap S_r\setminus\{q_{r+1}\}$.  Now (1) and (2) holds for $\{q_1,...,q_{r+1}\}$ and $S_{r+1}$.  In order to obtain (3), notice that we have the recursive formula

$$|S_{r+ 1}| \geq \frac{|S_r| -1 }{r^{d^2}}.$$

\noindent Substituting in the lower bound on $|S_r|$, we obtain the desired bound

$$|S_{r+ 1}| \geq \frac{N}{((r-1)!)^{d^2}r^{d^2}}- (r+1) = \frac{N}{(r!)^{d^2}}- (r+1).$$
This shows that we can construct the sequence $q_1,\ldots,q_{r+1}$ and the set $S_{r+1}$ with the three desired properties.  Since

$$|S_{2M}| \geq  \frac{2^{4d^2M\log M}}{((2M-1)!)^{d^2}}- 2M \geq 1,$$

\noindent this implies that the set $\{q_1,...,q_{2M}\}$ is well defined for $M = \ot_{d-1}(n-1)$.  By the pigeonhole principle, there exists a subset $Q \subset \{q_1,...,q_{2M}\}$ such that $|Q| \geq M = \ot_{d-1}(n-1)$, and $Q$ lies either above or below the point $q_{2M}$.  We will only consider the case when $Q$ lies below $q_{2M}$, since the other case is symmetric. We define the hyperplane $h = \{(x_1,...,x_d) \in \mathbb{R}^d: x_d = c\}$, where $c$ is a constant such that $h$ separates $Q$ and $q_{2M}$. For each point $q_i \in Q$, let $q_iq_{2M}$ be the line in $\mathbb{R}^d$ containing points $q_i$ and $q_{2M}$. Then we define the map $\phi:Q\rightarrow Q^{\ast}$, where $\phi(q_i) = q_iq_{2M}\cap h$.   With a slight perturbation of $Q$ if necessary, $Q^{\ast}$ is also in general position in $h = \mathbb{R}^{d-1}$, and the map $\phi$ is bijective.   See Figure \ref{project}.

\begin{figure}[h]
\begin{center}
\includegraphics[width=240pt]{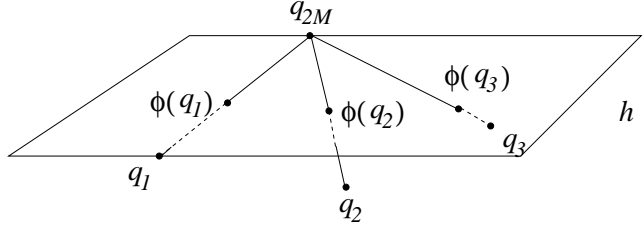}
  \caption{Mapping $\phi:Q\rightarrow Q^{\ast}$.}\label{project}
 \end{center}
\end{figure}

\noindent  By definition of $\ot_{d-1}(n-1)$, there exist $n-1$
points $\{q^{\ast}_{1},...,q^{\ast}_{n-1}\}\subset Q^{\ast}$ such that every $d$-tuple has the same orientation\footnote{With respect to the basis $e'_1 = \langle 1,0,...0,c\rangle, e'_2 = \langle 0,1,...0,c\rangle, \cdots , e'_{d-1}=\langle 0,0,...,1,c\rangle$} in $h = \mathbb{R}^{d-1}$.  This implies that $q_{2M}$ lies on the same side of each hyperplane spanned by $d$ points from $\{\phi^{-1}(q^{\ast}_{1}),\phi^{-1}(q^{\ast}_{2}),...,\phi^{-1}(q^{\ast}_{n-1})\}$ in $\mathbb{R}^{d}$. Hence every $(d+1)$-tuple of the form $(\phi^{-1}(q^{\ast}_{i_1}),...,\phi^{-1}(q^{\ast}_{i_d}),q_{2M})$, $1\leq i_1 < \cdots < i_d\leq n-1$, has the same orientation in $\mathbb{R}^d$.  By property (2), every $(d+1)$-tuple in the $n$-element set $\{\phi^{-1}(q^{\ast}_{1}),...,\phi^{-1}(q^{\ast}_{n-1}),q_{2M}\}\subset P$ has the same orientation.

\end{proof}

Theorem \ref{main1} now follows by applying Lemma \ref{hyperplane} with the fact that $\ot_2(n) = 2^{O(n)}$.

\section{Concluding remarks}

Let us remark that a simple modification to the construction of Eli\'a\v{s} and Matou\v{s}ek \cite{matousek} shows that $\ot_d(n) \geq 2^{2^{\Omega(n)}}$ for $d\geq 3$.  The best known estimates on $\ot_d(n)$ can be summarized in the following table.

\medskip

\medskip
\medskip
\begin{center}
    \renewcommand*\arraystretch{1.5}
 \begin{tabular}{|c|c|c|}
   \hline
   dimension $d$ & best results & references \\\hline
   $d = 1$  & $\ot_1(n) = (n-1)^2 + 1$  & Erd\H os and Szekeres \cite{es}\\
    $ d = 2$ &  $\ot_2(n) = 2^{\Theta(n)}$   &   Erd\H os and Szekeres \cite{es}  \\
   $d = 3$  & $\ot_3(n) = 2^{2^{\Theta(n)}}$  &  Eli\'a\v{s}-Matou\v{s}ek \cite{matousek} and Theorem \ref{main1}  \\
   $d\geq 4$ &  $2^{2^{\Omega(n)}} \leq \ot_d(n) \leq \twr_d(O(n))$  & Eli\'a\v{s}-Matou\v{s}ek \cite{matousek} and Theorem \ref{main1}  \\
   \hline
 \end{tabular}
 \end{center}

\medskip
\medskip
\medskip

  Hence, there is a significant gap between the known upper and lower bounds on $\ot_d(n)$ for $d\geq 4$.  We believe that $\ot_d(n)$ is on the order of $\twr_{d}(\Theta(n))$.


\begin{thebibliography}{99}


\bibitem{suk} D. Conlon, J. Fox, J. Pach, B. Sudakov, and A. Suk, Ramsey-type results on semi-algebraic relations, to appear in \emph{Transactions of the American Mathematical Society}.

\bibitem{duchet} R. Cordovil and P. Duchet, Cyclic polytopes and oriented matroids, \emph{European Journal of Combinatorics} \textbf{21} (2000), 49--64,


\bibitem{matousek} M. Eli\'a\v{s} and J. Matou\v{s}ek, Higher-order Erd\H os-Szekeres theorems, {Adv. Math.} \textbf{244} (2013), 1--15.


\bibitem{erdos2} P. Erd\H os, Some remarks on the theory of graphs, \emph{Bull. Amer. Math. Soc.} \textbf{53} (1947), 292--294.

    \bibitem{erdos3} P. Erd\H os, A. Hajnal, and R. Rado, Partition relations for cardinal numbers, \emph{Acta Math. Acad. Sci. Hungar.} \textbf{16} (1965), 93--196.

\bibitem{rado} P. Erd\H os and R. Rado, Combinatorial theorems on classifications of subsets of a given set, \emph{Proc. London Math. Soc.} \textbf{3} (1952), 417--439.


\bibitem{es} P. Erd\H os and G. Szekeres, A combinatorial problem in geometry, \emph{Compos. Math.} \textbf{2} (1935), 463--470.

\bibitem{es22}  P. Erd\H os and G. Szekeres, On some extremum problems in elementary geometry, \emph{Ann. Universitatis Scientiarum Budapestinensis, E\"otv\"os, Sectio Mathematica} \textbf{3/4} (1960-1961), 53--62.


\bibitem{gp} J. E. Goodman and R. Pollack, Allowable sequences and order types in discrete and computational geometry, In J.Pach editor, New Trends in Discrete and Computational Geometry, volume 10 of Algorithms and Combinatorics (1993), Springer, Berlin etc., 103--134.

\bibitem{graham} R. L. Graham, B. L. Rothschild, and J. H. Spencer, \emph{Ramsey Theory}, 2nd Edition, Wiley, New York, 1990.

\bibitem{grunbaum} B. Gr\"unbaum, \emph{Convex Polytopes}, 2nd edition, prepared by Volker Kaibel, Victor Klee, and Günter M. Ziegler, 2003.

\bibitem{hubard} A. Hubard, L. Montejano, E. Mora, and A. Suk, Order types of convex bodies, \emph{Order} \textbf{28} (2011), 121--130.

\bibitem{mat} J. Matou\v{s}ek, \emph{Lectures on Discrete Geometry}, Springer-Verlag New York, Inc., 2002.



\bibitem{mc} P. McMullen, The maximal number of faces of a convex polytope, \emph{Mathematika} \textbf{17} (1970), 179--184.

\bibitem{mosh}  G. Moshkovitz and A. Shapira, Ramsey-theory, Integer partitions and a new proof of the Erd\H os-Szekeres theorem, submitted.

\bibitem{steele} M.J. Steele, Variations on the monotone subsequence theme of Erd\H os and Szekeres.  In \emph{D. Aldous et al., eitors, Discrete Probability and Algorithms, IMA Volumes in Mathematics and its Applications} \textbf{72} (1995),Springer, Berlin etc., 111--131.



\end{thebibliography}
\end{document}